\documentclass[11pt,final]{amsart}

\title{On the $K$-theory of truncated polynomial algebras, revisited}
\author{Martin Speirs}

\usepackage{import}
\usepackage{ThesisPreamble}
\usepackage{verbatim}
\usepackage{soul}
\usepackage{booktabs}
\usepackage[square,sort,comma,numbers]{natbib}

\usepackage{todonotes}

\newcommand{\HH}{\operatorname{HH}}

\newcommand{\THH}{\operatorname{THH}}

\newcommand{\TP}{\operatorname{TP}}

\newcommand{\TC}{\operatorname{TC}}
\newcommand{\HC}{\operatorname{HC}}

\newcommand{\can}{\operatorname{can}}
\newcommand{\lax}{\operatorname{lax}}

\newcommand{\Bcy}{\operatorname{B^{cy}}}
\newcommand{\sd}{\operatorname{sd}}
\newcommand{\map}{\operatorname{map}}
\newcommand{\Htilde}{\tilde{H}}
\newcommand{\WW}{\mathbb{W}}
\newcommand{\CycSp}{\operatorname{CycSp}}

\begin{document}
\maketitle

\section{Introduction}

\noindent The algebraic $K$-theory of truncated polynomial algebras over perfect fields of positive characteristic was first evaluated by Hesselholt and Madsen \cite{HesselholtMadsenCyclicPoly}. Their proof relied on a delicate analysis of the facet structure of regular cyclic polytopes.
In this paper, we show that the Nikolaus-Scholze approach to topological cyclic homology \cite{NikolausScholze} now makes it possible to give a purely homotopy-theoretic proof of this result. In fact, the only input we use is the calculation of the homology of the cyclic bar construction together with the action of Connes' operator thereon. 

\begin{theorem}[\cite{HesselholtMadsenCyclicPoly} Theorem A]\label{KofTruncatedPoly}
Let $k$ be a perfect field of positive characteristic. Then there is an isomorphism
\[
K_{2r-1}(k[x]/(x^e),(x)) \simeq \WW_{re}(k) / V_e \WW_r(k)
\]
and the groups in even degrees are zero.
\end{theorem}

\begin{comment}
\noindent We present a calculation of the algebraic $K$-theory of truncated polynomial algebras over perfect fields of positive characteristic. This was first done by Hesselholt and Madsen in \cite{HesselholtMadsenCyclicPoly}. Their proof uses the cyclotomic trace map to topological cyclic homology and necessitates an understanding of the equivariant homotopy type of certain $\TT$-spaces arising as subspaces of the cyclic bar construction for a pointed monoid.  Applying recent progress on topological cyclic homology, due to Nikolaus and Scholze \cite{NikolausScholze}, we redo the computation using only the homology of these $\TT$-spaces, together with the action of Connes' operator. This is a first step towards making $\TC$ as easy to compute as Connes' cyclic homology $\HC$. 
\end{comment}

We briefly summarize the method. Let $k$ be a perfect field of characteristic $p > 0$ and let $A = k[x]/(x^e)$ and $I = (x)$ the ideal generated by the variable.  %McCarthy's theorem reduces the calculation of the relative $K$-groups $K(A,I)$ to the relative $\TC$-groups $\TC(A,I)$. 
The $k$-algebra $A$ is the pointed monoid algebra for the pointed monoid $\Pi_e = \{0, 1, x, \dots , x^{e-1}\}$ determined by $x^e = 0$. There is a canonical equivalence of cyclotomic spectra
\[
\THH(A) \simeq \THH(k) \otimes \Bcy(\Pi_e)
\]
where the Frobenius morphism on the right is the tensor product of the usual Frobenius and the unstable Frobenius on the cyclic bar construction of $\Pi_e$, see \cref{THHsection} for details.
Using the theory of cyclic sets one obtains a $\TT$-equivariant splitting of the cyclic bar construction,
\[
\Bcy(\Pi_e) \simeq \bigvee_{m \geq 0} B(m) 
\]
into simpler $\TT$-spaces $B(m)$. 
The singular homology and Connes' operator of these $\TT$-spaces is easily determined and reduces to computations of the Hochschild homology of $A$ first carried out in \cite{BACH91} and \cite{LarsenLindenstrauss}. The answer is simple enough that the Atiyah-Hirzebruch spectral sequence degenerates allowing us to directly determine the homotopy groups of $\THH(k) \otimes B(m)$. 
From \cite{NikolausScholze}  the topological cyclic homology of $A$ is given by the equalizer
\[
\TC(A;p) \to \TC^{-}(A) \overset{\phi - \can}\longto \TP(A)
\]
so using the above splitting this reduces to computing $(\THH(k) \otimes B(m))^{h\TT}$ and $(\THH(k) \otimes B(m))^{t\TT}$. We achieve this by an inductive procedure, making use of the highly co-connective Frobenius map 
\[
\phi : (\THH(k) \otimes B(m))^{h\TT} \to (\THH(k) \otimes B(pm))^{t\TT}
\]
and the periodicity of $(\THH(k) \otimes B(m))^{t\TT}$. Assembling the answers for varying $m$ then yields the $\TC$-calculation. 
Applying McCarthy's theorem one obtains the result.

We remark that the method employed here was also used by Hesselholt and Nikolaus  \cite{HesselholtNikolausCusp} to evaluate the $K$-theory of cuspidal curves over $k$, thereby affirming the conjectural calculation in \cite{HesselholtCusps}. We consider this method a first step towards making topological cyclic homology as easy to compute as Connes' cyclic homology $\HC$. 

\subsection{Acknowledgements} 
I would like to thank Lars Hesselholt for his generous and valuable guidance while working on this project. I would also like to thank Christian Ausoni, Ryo Horiuchi, Malte Leip and Joel Stapleton for several useful conversations and comments during the production of this paper. I thank the anonymous referee for several helpful comments on content and exposition.
This paper is based upon work supported by the DNRF Niels Bohr Professorship of Lars Hesselholt
as well as the National Science Foundation under Grant No.\ DMS-1440140 while the author was in residence at the Mathematical Sciences Research Institute in Berkeley, California during the Spring 2019 semester.

\section{Witt vectors, big and small}
\noindent
The purpose of this short section is to show the following well-known splitting. Let $s = s(p,r,d)$ be the unique positive integer such that 
\[
p^{s-1}d \leq r < p^sd
\]
if it exists, or else $s=0$. 

\begin{lemma}\label{WittSplit}
Let $k$ be a perfect field of characteristic $p>0$. Let $e = p^ue'$ with $(p,e') = 1$. There is an isomorphism
\[
\WW_{re}(k)/V_e\WW_r(k) \simeq \prod W_h(k)
\]
where the product is indexed over $1 \leq m' \leq re$ with $(p,m')=1$ and with $h = h(p,r,e,m')$ given by
\[
h =  \left\{
	\begin{array}{ll}
		  s & \mbox{if } e' \nmid m' \\
		  \min\{u,s\} & \mbox{if } e' \mid m'
	\end{array}
\right.
\]
where $s = s(p,re,m')$ is the function defined above.
\end{lemma}
\begin{proof}
We use the isomorphism
\[
\WW_r(k) \overset{\langle I_d \rangle}\longto \prod W_s(k)
\]
which is a map of $\ZZ_{(p)}$-algebras, where the product runs over $d$ such that $(p,d)=1$ and $1 \leq d \leq r$ and where $s = s(p,r,d)$, see for example \cite[Prop. 1.10 and Example 1.11]{HesselholtBigDRW}. The $d$'th component of this map is the composite
\[
I_d : \WW_r(k) \overset{F_d}\to \WW_{\lfloor r/d \rfloor}(k) \overset{R}\to W_s(k)
\]
where $F_d$ is the Frobenius map and $R$ is the restriction map induced by the inclusion $\{1,p,p^2, \dots, p^{s-1}\} \subs \{1,2, \dots, \lfloor r/d \rfloor \} $. If $m' = e'd$ with $d \leq r$ then one readily checks that $s(p,re,m') = s(p,r,d) + u$. Furthermore in this case the following diagram commutes
\begin{equation*}
\begin{tikzcd}[row sep=3em, column sep=3em]
\WW_r(k) \arrow[r, "I_d"] \arrow[d, "V_e"] & W_s(k) \arrow[d, "e'V_{p^u}"] \\
\WW_{re}(k) \arrow[r, "I_{m'}"] & W_{s+u}(k)
\end{tikzcd}
\end{equation*}
Indeed this may be checked after applying the ghost map, where it is a routine verification analogous to \cite[Lemma 1.5]{HesselholtBigDRW}. This corresponds to the case $u \leq s$.
Since $(p,e')=1$ we have
\[
W_{s+u}(k)/(e'V_{p^u}W_s(k)) \cong W_u(k) .
\]
Thus, we get an isomorphism
\[
\WW_{re}(k)/V_e\WW_r(k) \overset{\simeq}\to \prod W_u(k) \times \prod W_s(k) \times \prod W_s(k) \overset{\simeq}\to \prod W_h(k)
\]
where in the middle term, the first product is indexed over $1 \leq d \leq r$ with $(p,d)=1$, the second product is indexed over $1 \leq m' \leq re$ with $e' \mid m'$ and with $u > s$, the third product is indexed over $1 \leq m' \leq re$ with $e' \nmid m'$  and with $(p,m')=1$. In the last term, the product is indexed over $1 \leq m' \leq re$ with $(p,m')=1$.
\end{proof}

%%%%%%%%%%%%%%%%%%%%%%%%%%%%%%%%%%%%%%%%
%%%%%%%%%%%%%%%%%%%%%%%%%%%%%%%%%%%%%%%%
%%%%%%%%%%%%%%%%%%%%%%%%%%%%%%%%%%%%%%%%

\section{Hochschild homology of truncated polynomial algebras}
\noindent
In this section we review the results of \cite{BACH91} and \cite{LarsenLindenstrauss} on cyclic homology of algebras of the form $A = k[x]/f(x)$.
We work over a general commutative unital base ring $k$. The Hochschild homology of $A$ over $k$ is the homology of the associated chain complex for the cyclic $k$-module
\[
\Bcy(A/k)[n] = A^{\otimes n+1}
\]
where the tensor product is over $k$. The cyclic structure maps are given as follows
\begin{eqnarray*}
d_i(a_0 \otimes \dots \otimes a_n) &=& \left\{
        \begin{array}{ll}
         a_0 \otimes \dots \otimes a_i a_{i+1} \otimes \dots \otimes a_{n}    & \quad 0 \leq i < n \\
          a_{n}a_0 \otimes a_1 \otimes \dots  \otimes a_{n-1} & \quad  i = n
        \end{array}
    \right.
\\
s_i(a_0 \otimes \dots \otimes a_n) &=&  a_0 \otimes \dots \otimes a_i \otimes 1 \otimes a_{i+1} \otimes \dots \otimes a_n
\\
t_n(a_0 \otimes \dots \otimes a_n) &=& a_n \otimes a_0 \otimes \dots \otimes a_{n-1} .
\end{eqnarray*}

The Hochschild homology then is the homology $\HH_*(A/k)$ of the associated chain complex with differential given by the alternating sum of the face maps.

\begin{proposition}\label{HHcomputation}
Let $A = k[x]/(x^e)$ where $k$ is a commutative unital ring. There is an isomorphism
\[
\HH_*(A/k) =  \left\{
	\begin{array}{lll}
		  A & \mbox{if } * = 0 \\
		  k[e]\{1\} \oplus k\{x, \dots , x^{e-1} \} & \mbox{if } * > 0 \text{ even} \\
		  k\{1,x, \dots , x^{e-2}\} \oplus k/ek\{x^{e-1}\} & \mbox{if } * > 0 \text{ odd}
	\end{array}
\right.
\]
where $k[e]$ denotes the $e$-torsion elements of $k$.
\end{proposition}

The proof uses a common technique for such rings, namely the construction of a small and computable complex. The task is then to show that this complex is quasi-isomorphic to the Hochschild complex. For a $k$-algebra $A$ of the form $A = k[x]/(f(x))$, assuming it is flat as an $k$-module then the Hochschild homology may be calculated as $\Tor^{A^e}_*(A,A)$ where $A^e = A \otimes A^{op}$. So it suffices to find a small $A{-}A$-bimodule resolution of $A$. Given such a resolution $R(A)_* \to A$ one now tensors over $A^e$ with $A$ to get a complex, $\overline{R}(A)_*$ computing $\HH_*(A/k)$.  For an appropriate choice of resolution the corresponding complex $\overline{R}(A)_*$ has the following form
\[
0 \leftarrow A \overset{0}\leftarrow  A \overset{f'(x)}\leftarrow A  \overset{0}\leftarrow A  \overset{f'(x)}\leftarrow A \leftarrow \cdots
\]
from which the result readily follows. 

\begin{comment}
One may then express Connes' $B$ operator in terms of this smaller complex. Let $d_A : A \to A$ be $g(x) \mapsto g'(x)$ and let $d\log f(x) : A \to A$ be given by 
\[
g(x) \mapsto \frac{g(x) f'(x)}{f(x)}
\]
then $B_{small} : A \to A$ is given by the zero map in odd degrees, and $d_A + n d\log f(x)$ in even degree $2n$. So we get the mixed complex
\end{comment}

We now introduce a splitting of the Hochschild homology of the $k$-algebra $A = k[x]/(x^e)$. Equip $A$ with a ``weight'' grading by declaring $x^m$ have weight $m$. This induces a grading on the tensor powers of $A$ and we let
\[
\Bcy(A/k;m)[n] \subs \Bcy(A/k)[n] = A^{\otimes n+1}
\]
be the sub $k$-module of weight $m$. It is generated by those tensor monomials whose weight is equal to $m$. This forms a sub cyclic $k$-module of  $\Bcy(A/k)[-]$ and so we obtain a splitting
\[
\Bcy(A/k)[-] \simeq \bigoplus_{m \geq 0} \Bcy(A/k;m)[-]
\]
of cyclic $k$-modules, and of the associated chain complexes. Taking homology then gives a splitting as well,
\[
\HH_*(A/k) \simeq \bigoplus_{m \geq 0} \HH_*(A/k;m).
\]
In the following lemma, let $d = d(e,m) = \lfloor \frac{m-1}{e} \rfloor$ be the largest integer less that $(m-1)/e$ . %Thus $d$ is the largest possible quotient of $m$ by $e$, i.e. $m = t + ed$ for a unique $0 \leq t < e$.

\begin{lemma}\label{HHweightComputation}
Let $k$ and $A$ be as in \cref{HHcomputation}. If $m$ is not a multiple of $e$ then $\HH_*(A/k;m)$ is concentrated in degrees $2d$ and $2d+1$ where it is free of rank $1$ as a $k$-module. In this case Connes' $B$-operator takes the generator in degree $2d$ to $m$ times the generator in degree $2d+1$, up to a sign. If $m$ is a multiple of $e$ then $\HH_*(A/k;m)$ is concentrated in degree $2d+1$ and $2d+2$. The group in degree $2d+1$ is isomorphic to $k/ek$ while the group in degree $2d+2$ is isomorphic to $k[e]$. In this case Connes' operator acts trivially. 

\end{lemma}
\begin{proof}
First we prove that the groups are as stated. We follow the proof given in \cite[Section 7.3.]{HesselholtMadsen97}. Consider the resolution of $A$ as an $A\otimes A$-module constructed by \cite{BACH91}, denoted $R(A/k)_*$ having the form
\[
\cdots \overset{\Delta}\longto A \otimes A \overset{\delta}\longto A\otimes A  \overset{\Delta}\longto A \otimes A  \overset{\delta}\longto A \otimes A  \overset{\mu}\longto A \to 0
\]
where 
\[
\Delta = \frac{x^e\otimes 1 - 1 \otimes x^e}{x\otimes 1 - 1 \otimes x} 	\quad \text{and} \quad \delta = 1 \otimes x - x \otimes 1
\]
In \cite{BACH91} a quasi-isomorphism $\psi : R(A/k)_* \longto B(A/k)_*$ with the bar-resolution is constructed. Since $\Delta$ increases the weight by $e-1$ and $\delta$ by $1$, and since the differential $b'$ of the bar resolution preserves weight, we see (by induction on $j$) that $\psi_{2j}$ increases weight by $je$, whereas $\psi_{2j+1}$ increases weight by $je+1$. Tensoring over $A^e$ with $A$ gives a quasi-isomorphism $\overline{\psi} : \overline{R}(A/k)_* \longto \Bcy(A/k)_*$ which has the same weight shift. 
%As a result $\HH_{*}(A/k;m) = H_{*}(A(t))$ where $t$ is the unique number $0 \leq t < e$ such that $m = t + ed$. 
The result now follows from \cref{HHcomputation}

For the statements about Connes' operator, this follows by an explicit choice of a quasi-isomorphism $\psi$ (and its inverse). This is done in \cite[Section 1]{BACH91} and in \cite[Proposition 2.1.]{BACH91} the computation of Connes' operator is given.
\end{proof}

%%%%%%%%%%%%%%%%%%%%%%%%%%%%%%%%%%%%%%%%
%%%%%%%%%%%%%%%%%%%%%%%%%%%%%%%%%%%%%%%%
%%%%%%%%%%%%%%%%%%%%%%%%%%%%%%%%%%%%%%%%

\section{Topological Hochschild homology and the cyclic bar construction} \label{THHsection}
\noindent
Let $\Pi_e = \{0,1,x, \dots, x^{e-1}\}$ be the pointed monoid determined by setting $x^e = 0$. Then the truncated polynomial algebra $A$ is the pointed monoid ring $k(\Pi_e) = k[\Pi_e]/k[0]$. The cyclic bar construction of $\Pi_e$ is the cyclic set $\Bcy(\Pi_e)[-]$ with
\[
\Bcy(\Pi_e)[k] = \Pi_e^{\wedge (k+1)}
\]
and with the usual Hochschild-type structure maps. We write $\Bcy(\Pi_e)$ for the geometric realization of $\Bcy(\Pi_e)[-]$. The space $\Bcy(\Pi_e)$ admits a natural $\TT$-action where $\TT$ is the circle group, as does the geometric realization of any cyclic set. Furthermore it is an unstable cyclotomic space, i.e.\ there is a map
\[
\psi_p : \Bcy(\Pi) \to \Bcy(\Pi)^{C_p}
\]
which is equivariant when the domain is given the natural $\TT/C_p$-action. For a construction of this map see \cite[Section 2]{BHM} or, for a review in our setup, see \cite[Section 3.1.]{SpeirsCoordinateAxes}. See also \cite[Section IV.3]{NikolausScholze} for similar constructions in the context of $\mathbb{E}_1$-monoids in spaces.

To every non-zero $n$-simplex $\pi_0 \wedge \dots \wedge \pi_n \in \Bcy(\Pi_e)[n]$ we associate its \emph{weight} as follows, each $\pi_i$ is equal to $x^{m_i}$ for some $0 \leq m_i \leq e-1$. Let 
\[
w(\pi_0 \wedge \dots \wedge \pi_n) = \sum_{i=0}^n m_i .
\]
The weight is preserved by the cyclic structure maps and so we obtain a splitting of pointed cyclic sets 
\[
\Bcy(\Pi_e)[-] = \bigvee_{m \geq 0} \Bcy(\Pi_e;m)[-]
\]
where $ \Bcy(\Pi_e;m)[-] \subs \Bcy(\Pi_e)[-]$ consists of all simplices with weight $m$. Let $B(m)$ denote the geometric realization of $\Bcy(\Pi_e;m)[-]$. So we have a splitting of pointed $\TT$-spaces
\begin{equation}\label{CyclicBarDecomposition}
\Bcy(\Pi_e) \simeq \bigvee_{m \geq 0} B(m) .
\end{equation}
By \cite[Splitting lemma]{SpeirsCoordinateAxes} we have $\THH(k(\Pi_e)) \simeq \THH(k) \otimes \Bcy(\Pi_e)$ as cyclotomic spectra. Here the Frobenius on the right hand side is the tensor product of the usual Frobenius on $\THH(k)$ (as constructed in \cite[Section III.2]{NikolausScholze}) and the Frobenius on $\Sigma^\infty\Bcy(\Pi_e)$ arising from the unstable Frobenius (see \cite[Section on cyclic bar construction]{SpeirsCoordinateAxes}). We are interested in the \emph{relative} $\THH$, defined for any ring $A$ and ideal $I$ as the homotopy fiber $\THH(A,I)$ of the map
\[
\THH(A) \to \THH(A/I)
\]
induced by the quotient map. 
In the case at hand, the relative $\THH$ corresponds to simply cutting out the weight zero part, i.e.\ we have an equivalence of spectra with $\TT$-action
\[
\THH(A,I) \simeq \bigoplus_{m \geq 1} \THH(k) \otimes B(m) 
\]
where $I = (x)$ is the ideal generated by the variable. To see this note that the composite of the canonical map 
\[
\THH(k) \otimes B(m) \to \THH(A)
\]
with the map $\THH(A) \to \THH(k)$ is constant for $m \geq 1$ and is the identity map for $m=0$ where $B(0) = S^0$.

Given any pointed monoid $\Pi$ there is an isomorphism of cyclic $k$-modules
\[
w : k(\Bcy(\Pi)[-]) \longto \Bcy(k(\Pi)/k)[-]
\]
which maps $\pi_0 \wedge \dots \wedge \pi_n$ to $ \pi_0 \otimes \dots \otimes \pi_n$. Note that $k(\Bcy(\Pi)[-])$ is the simplicial complex for the space $\Bcy(\Pi)$. In particular the associated homology $H_*(k(\Bcy(\Pi)[-]))$ computes the simplicial homology of $\Bcy(\Pi)$.

In the following lemma, we let $d = d(e,m) = \lfloor \frac{m-1}{e} \rfloor$ for any $m \geq 1$.

\begin{lemma}(\rm{\cite[Lemma 7.3]{HesselholtMadsen97}})\label{HomologyLemma}
Let $A = k[x]/(x^e)$ where $k$ is a commutative unital ring and let $B(m) \subs \Bcy(\Pi_e)$ be as described above.
\begin{enumerate}
\item If $e \nmid m$ then $\Htilde_*(B(m);k)$ is free of rank $1$ if $* = 2d, 2d+1$ and trivial, otherwise. The Connes' operator takes a generator in degree $2d$ to $m$ times a generator in degree $2d+1$.
\item If $e \mid m$ then $\Htilde_*(B(m);k)$ is isomorphic to $k/ek$  if $* = 2d+1$, to $k[e]$ if $*=2d+2$, and trivial otherwise.
\end{enumerate}
\end{lemma}
\begin{proof}
We use the isomorphism of cyclic $k$-modules 
\[
w : k(\Bcy(\Pi_e)[-]) \to \Bcy(A/k)[-]
\]
This map preserves the weight decomposition, mapping $k(B(m)[-])$ isomorphically to $\Bcy(A/k;m)[-]$. Furthermore the map commutes with the Connes operator, as shown in the proof of \cite[Proposition 1.4.5.]{Hesselholt96}. Now by \cref{HHweightComputation} we can read off what 
\[
\HH_*(A/k;m) = \Htilde_*(B(m);k)
\]
is and how Connes' operator acts. 
\end{proof}

Note that in particular when $e$ is zero in $k$, $\Htilde_{2d+2}(B(m);k)$ is free of rank $1$ over $k$. Thus there is room for a non-trivial Connes' operator in this case. However, it follows again from \cref{HHweightComputation} that it is trivial in this case. 

We need the following general commuting diagram that we state as a lemma.  Let $G$ be a group and $f : BG \to *$ be the projection to a point. The induced pullback functor $f^* : \Sp \to \Sp^{BG}$ admits a right adjoint 
\[
f_* = (-)^{hG} : \Sp^{BG} \longto \Sp,
\]
given by the limit functor $\lim_{BG}(-) : \Sp^{BG} = \mathrm{Fun}(BG,\Sp) \to \Sp$, \emph{cf.}~ \cite[Section I.1]{NikolausScholze}. We denote by $\epsilon : f^*f_* \to \id$ the counit of this adjunction. Following \cite[Theorem I.4]{NikolausScholze} we denote by 
\[
f_*^T : \Sp^{BG} \to \Sp
\]
the corresponding Tate $G$-construction. There is a canonical map $\alpha^T : f_*^TX \otimes f_*Y \to f_*^T(X \otimes f^*f_*Y)$ which may be defined by its adjoint: the initial map
\[
f_*^TX \longto \map(f_*Y, f_*^T(X \otimes f^*f_*Y)) 
\]
such that precomposing with
 \[
\can : f_*X \to f_*^TX
\]
 gives the composite
\[
f_*X \longto \map(f_*Y, f_*(X \otimes f^*f_*Y)) \overset{\can_*}\longto \map(f_*Y, f_*^T(X \otimes f^*f_*Y)) 
\]
(\emph{cf}.\ \cite[Section I.3]{NikolausScholze}).
\begin{lemma}\label{LaxTateDiagram}
Let $X$ and $Y$ be spectra with $G$-action. Then the following diagram commutes.
\[
\begin{tikzcd}[row sep=4em, column sep=4em]
f_*^TX \otimes f_*Y \arrow[r, "\id \otimes \can"] \arrow[d, "\alpha^T"]  & f_*^TX \otimes f_*^TY \arrow[d, "\lax"]  \\
f_*^T(X \otimes f^*f_*Y) \arrow[r,"f_*^T(\epsilon)"] & f_*^T(X \otimes Y)
\end{tikzcd}
\]
\end{lemma}
\begin{proof}
Consider the adjoint maps 
\[
\alpha, \beta : f_*^TX \longto \map(f_*Y , f_*^T(X \otimes Y))
\]
determined by the lower and upper composite, respectively. It is enough to check that the maps agree after precomposing with the canonical map $f_*X \to f_*^TX$. This follows by the construction of $\alpha^T$ and by the lax symmetric monoidality of $\can_*$ \cite[Theorem I.4.1.(vi)]{NikolausScholze}.
\end{proof}

As observed in  \cite[Section IV.2]{NikolausScholze},  $\THH$ naturally forms a lax symmetric monoidal functor $\mathrm{Alg}_{E_1}(\Sp) \to \CycSp$ which is in fact symmetric monoidal, i.e.\ the lax structure map is an equivalence. For this last claim it is enough to check the equivalence on the underlying map of spectra for which see \cite[Theorem 14.1]{BlumbergMandell2017}. This symmetric monoidal structure of $\THH$ together with the $\TT$-equivariant decomposition \cref{CyclicBarDecomposition} provides us with the equivalence
\[
\THH(k(\Pi_e)) \simeq \bigoplus_{m \geq 0} \THH(k) \otimes B(m) .
\]
of spectra with $\TT$-action. We wish to identify the right hand side as a cyclotomic spectrum. By definition of the symmetric monoidal structure on cyclotomic spectra the Frobenius map in questions factors as
\begin{eqnarray*}
\THH(k) \otimes B(m) & \overset{\phi \otimes \widetilde{\psi}}\longto &  \THH(k)^{tC_p} \otimes B(pm)^{hC_p}
\\ & \overset{id \otimes \can}\longto &  \THH(k)^{tC_p} \otimes B(pm)^{tC_p}
\\ & \longto & (\THH(k)\otimes B(pm))^{tC_p}
\end{eqnarray*}
where the final map is the lax symmetric monoidal structure map for the Tate-$C_p$-construction. There is a unique such map making the natural transformation $\can : (-)^{hC_p} \to (-)^{tC_p}$ lax symmetric monoidal \cite[Theorem I.3.1]{NikolausScholze}.  We may factor the Frobenius
\[
\THH(k) \otimes B(m) \overset{\phi \otimes id}\longto \THH(k)^{tC_p} \otimes B(m) \overset{}\longto (\THH(k) \otimes B(pm))^{tC_p}
\]
as the Frobenius on $\THH(k)$ followed by the map induced by the unstable Frobenius $\widetilde{\phi} : B(m) \to B(pm)^{hC_p}$.

\begin{lemma}\label{FrobeniusLemma}
The map
\[
\THH(k)^{tC_p} \otimes B(m) \longto (\THH(k) \otimes B(pm))^{tC_p}
\]
induced by the unstable Frobenius is an equivalence.
%\begin{eqnarray*}
% \THH(k) \otimes B(m) 
%&\overset{\phi \otimes \psi_p}{\longto}  \THH(k)^{tC_p} \otimes B(pm)^{C_p}
%\\ &\overset{\simeq}\longto   (\THH(k) \otimes B(pm)^{C_p})^{tC_p} 
%\\ &\overset{\simeq}\longto  \left( \THH(k) \otimes B(pm)  \right)^{tC_p} 
%\end{eqnarray*} %\todo{identify the Frobenius}
\end{lemma}
\begin{proof}
We abbreviate $T = \THH(k)$. 
 Consider the following diagram
\[
\begin{tikzcd}[row sep=4em, column sep=4em]
T^{tC_p} \otimes B(m) \arrow[r, "\id \otimes \widetilde{\phi}"] \arrow[d, "\alpha^T"] & T^{tC_p} \otimes B(pm)^{hC_p} \arrow[d, "\alpha^T"] \arrow[r, "\id \otimes \can"] & T^{tC_p} \otimes B(pm)^{tC_p} \arrow[d, "\mathrm{lax}"] \\
(T \otimes B(m))^{tC_p} \arrow[r, "\widetilde{\phi}^{tC_p}"] & (T \otimes B(pm)^{hC_p})^{tC_p} \arrow[r, "\epsilon^{tC_p}"] & ( T \otimes B(pm))^{tC_p}
\end{tikzcd}
\]
where $\epsilon$ is the canonical equivariant map from the fixed points of a space (with trivial $C_p$-action) to itself, and ``lax'' is the lax symmetric monoidal structure map. The map $\alpha^T$ was constructed in \cref{LaxTateDiagram}.  We claim this diagram commutes and that the left-most vertical map, as well as the bottom composite, are equivalences. Here we equip $B(m)$ with trivial $C_p$-action. Since this is a finite space the map $\alpha^T$ is an equivalence. The bottom row is the map induced by the composite
\[
B(m) \overset{\tilde{\phi}}\longto B(pm)^{hC_p} \overset{\epsilon}\longto B(pm)
\]
where the first and second terms are given trivial $C_p$-action. The cofiber is a finite colimit of free $C_p$-cells hence  applying the Tate construction induces an equivalence, see for example \cite[Lemma I.3.8.]{NikolausScholze}.

The commutativity of the left-square follows from the naturality of the map $\alpha^T$. Finally the left-square of the diagram commutes by \cref{LaxTateDiagram}. 
\end{proof}

\begin{comment} %%% alternative proof using Segal conjecture. 
\begin{proof}
The proof follows that of the similar statement in \cite[Proposition 3.5.1]{SpeirsCoordinateAxes}. Taking $T = \THH(k)$ and $X = B(m)$ in \cref{LaxEquivalenceLemma} yields the claim about the lax symmetric monoidal structure map. To see that the restricted Frobenius is a $p$-adic equivalence, one factors it accordingly as follows.
\begin{center}
\begin{tikzcd}[row sep=4em, column sep=2em]
\SSS \otimes B(m) \arrow[d, "\Delta_p \otimes \tilde{\Delta_p}"] \arrow[rr, "\tilde{\phi}"] &  & (\SSS \otimes B(pm))^{tC_p} \\
\SSS^{tC_p} \otimes (\sd_p B(pm))^{C_p} \arrow[r, "(1.)"] & (\SSS \otimes \sd_p B(pm)^{C_p})^{tC_p} \arrow[r, "D_p"] & (\SSS \otimes B(pm)^{C_p})^{tC_p} \arrow[u, "(2.)"]
\end{tikzcd}
\end{center}
Now the Segal conjecture says that $\Delta_p : \SSS \to \SSS^{tC_p}$ is a $p$-adic equivalence. The map labelled $(1.)$ is an equivalence since $\sd_p B(pm)^{C_p}$ is a finite $C_p$-CW-complex with trivial $C_p$-action, and since $(-)^{tC_p}$ is exact. The map $D_p$ is the equivalence from the $p$-subdivision of $B(pm)$ to $B(pm)$ itself. Finally, the map labelled $(2.)$ is an equivalence since $(-)^{tC_p}$ is trivial on free $C_p$-CW-complexes, cf. \cite[Lemma 9.1.]{HesselholtMadsen97}.
\end{proof}
\end{comment}

\begin{corollary}\label{FrobeniusCorollary}
The restricted Frobenius map
\[
\phi(m) : \THH(k) \otimes B(m) \to (\THH(k) \otimes B(pm))^{tC_p}
\]
induces an isomorphism in degrees $\geq 2d+1$ when $e \nmid m$, and induces an isomorphism in degrees $\geq 2d+2$ when $e \mid m$.
\end{corollary}
\begin{proof}
This follows readily from \cref{FrobeniusLemma} and \cref{HomologyLemma} using the Atiyah-Hirzebruch spectral sequence and the fact that the Frobenius 
\[
\phi : \THH(k) \to \THH(k)^{tC_p}
\]
 is an equivalence on connective covers (see \cite[Corollary IV.4.13]{NikolausScholze} for $k=\FF_p$, and \cite[Addendum 5.3]{HesselholtMadsen97} for any perfect field $k$).
\end{proof}

%%%%%%%%%%%%%%%%%%%%%%%%%%%%%%%%%%%%%%%%
%%%%%%%%%%%%%%%%%%%%%%%%%%%%%%%%%%%%%%%%
%%%%%%%%%%%%%%%%%%%%%%%%%%%%%%%%%%%%%%%%

\section{Negative and periodic topological cyclic homolgy}
\noindent
In this section we compute the periodic and negative topological cyclic homology of the ring of truncated polynomials over a perfect field of characteristic $p>0$. We will use the homotopy fixed point spectral sequence and the Tate spectral sequence, which we briefly recall. Let $X$ be a  connective spectrum with $\TT$-action.  
The homotopy fixed point spectral sequence is a second quadrant spectral sequence converging to $\pi_*(X^{h\TT})$ and with $E^2$-page given by 
\[
E^2 = H^*(B\TT, \pi_*X) \simeq S_{\ZZ}\{t\} \otimes \pi_*(X)
\]
where $t$ has bidegree $(-2,0)$. Note that the $\TT$-action on $\pi_*(X)$ is necessarily trivial since $\TT$ is path-connected. The Tate spectral sequence is a half-plane conditionally convergent spectral sequence converging to $\pi_*(X^{t\TT})$ whose $E^2$-page is given by inverting $t$, i.e.
\[
E^2 = S_{\ZZ}\{t,t^{-1}\} \otimes \pi_*(X) 	\quad \Rightarrow \quad \pi_*(X^{t\TT}) .
\] 
Here $S_{\ZZ}\{t,t^{-1}\}$ is the Laurent polynomial algebra over $\ZZ$ on a generator $t$ with bidegree $(-2,0)$. It $X$ is a ring-spectrum with $\TT$-action, then both spectral sequences are multiplicative. See \cite[Section 4]{HesselholtMadsen2003} for the construction and basic properties of the Tate spectral sequence. See also the forthcoming \cite{HedenlundKrauseNikolaus} for a construction of the Tate spectral sequence in the context of the $\infty$-category of parametrized spectra.  We will also repeatedly use the following formula for the differentials on the $E^2$-page.

\begin{lemma}\label{SecondTateDifferential}
Let $X$ be a spectrum with $\TT$-action such that the underlying spectrum is an $H\ZZ$-module . The $d^2$ differential of the Tate spectral sequence is given by $d^2(\alpha) = t d(\alpha)$ where $d$ is Connes' operator.
\end{lemma}
\begin{proof}
See \cite[Lemma 1.4.2]{Hesselholt96} or \cite[Section on Tate spectral sequence]{HedenlundKrauseNikolaus}.  %[or B\"oksted-Madsen section 5?]
\end{proof}
\noindent For a ring $A$, one defines the negative (respectively, periodic) topological cyclic homology $\TC^-(A)$ (respectively, $\TP(A)$) by taking the homtopy fixed points (respectively, Tate constuction) on the spectrum with $\TT$-action $X = \THH(A)$.

Returning to the ring of truncated polynomials, we will compute $\TP$ and $\TC^{-}$ using an inductive procedure based on the $p$-adic valuation of the integer $m$ indexing the $\TT$-space $B(m)$. We choose generators for the homology of the spaces $B(m)$ following \cref{HomologyLemma}. If $e \nmid m$ let $y_m$ and $z_m$ be generators for the homology in degree $2d$ and $2d+1$, respectively. In this case the $E^2$-page of the Tate spectral sequence calculating $\pi_*((\THH(k) \otimes B(m))^{t\TT})$ is given by
\[
E^2 = k[t^{\pm 1},x]\{y_m,z_m\}
\]
where $y_m$ and $z_m$ have bidegrees $(0,2d)$ and $(0,2d+1)$ respectively.
If $e \mid m$ and $p \mid e$ then we let $z_m$ and $w_m$ be generators of the homology in degree $2d+1$ and $2d+2$, respectively. Then the $E^2$-page of the Tate spectral sequence is given by
\[
E^2 = k[t^{\pm 1},x]\{z_m,w_m\}
\]
where $z_m$ and $w_m$ have bidegrees $(0,2d+1)$ and $(0,2d+2)$ respectively.

Before stating the next lemma we need to introduce an important tool, namely a $\TT$-equivariant map (in fact it is a map of $p$-cyclotomic spectra) 
\[
H\ZZ_p \to \THH(k)
\]
Here $H\ZZ_p$ is given the trivial $\TT$-action. To get this map we use the calculation $\tau_{\geq 0} \TC(k) \simeq H\ZZ_p$ \cite[Theorem B]{HesselholtMadsen97} giving the $\TT$-equivariant map 
\[
H \ZZ_p \simeq \tau_{\geq 0} \TC(k) \to \TC(k) \to \TC^-(k) \to \THH(k).
\] 

\begin{lemma}\label{InfiniteCyclesLemma}
In the Tate spectral sequence converging to $\pi_*(( \THH(k) \otimes B(m))^{t\TT} )$ the class $z_m$ is an infinite cycle for all $m$.
\end{lemma}
\begin{proof}
Although the statement does not seem to require it, we must deal with the cases $e \mid m$ and $e \nmid m$ separately. In both cases we use the $\TT$-equivariant map $H\ZZ_p \to \THH(k)$ constructed in the preceding paragraph. 
This map induces a map from the Tate spectral sequences computing $\pi_*((H\ZZ_p \otimes B(m))^{t\TT})$ to the Tate spectral sequence computing $\pi_*((\THH(k) \otimes B(m))^{t\TT})$.

Suppose first that $e \nmid m$. Then from \cref{HomologyLemma} we may compute the $E^2$-page of the Tate spectral sequence for $H\ZZ_p  \otimes B(m)$ to be
\[
E^2 =\ZZ_p[t^{\pm 1}]\{y_m, z_m\} \quad \Rightarrow \quad \pi_*(H\ZZ_p \otimes B(m))^{t\TT}
\]
where $|y_m| = (0,2d)$ and $|z_m| = (0,2d+1)$. The differential structure is determined by $d^2(y_m) = mtz_m$ (using \cref{SecondTateDifferential}), and so $E^3 = E^\infty = \ZZ_p/m\ZZ_p[t^{\pm 1}]\{z_m\}$ so $z_m$ is an infinite cycle. It follows that $z_m \in k[t^{\pm 1}, x]\{y_m, z_m\}$ (the $E^2$ page for the target spectral sequence) is an infinite cycle.

Now suppose $e \mid m$. Then from \cref{HomologyLemma} we may compute the $E^2$-page of the Tate spectral sequence for $H\ZZ_p  \otimes B(m)$ to be
$\ZZ_p/e\ZZ_p[t^{\pm 1}]\{z_m\}$
with $|z_m| = (0,2d+1)$ from which it follows immediately that $z_m$ is an infinite cycle.
\end{proof}

%Let $d(e, m) = \lfloor \frac{m-1}{e} \rfloor$.

\begin{proposition}\label{TPordinary}
Write $m = p^vm'$ where $(m',p)=1$. If $e \nmid m$ then 
\[
\pi_{2r+1}(\THH(k) \otimes B(p^vm'))^{t\TT} \simeq W_v(k)
\]
for all $r \in \ZZ$, and
\[
\pi_{2r+1}(\THH(k) \otimes B(p^vm'))^{h\TT} \simeq  \left\{
	\begin{array}{ll}
		  W_{v+1}(k) & \mbox{if } d \leq r \\
		  W_v(k) & \mbox{if } r < d
	\end{array}
\right.
\]
The even homotopy groups are trivial.
\end{proposition}
\noindent In the following proof, and the rest of the paper, a dot above an equality indicates that the equality holds up to a unit. 
\begin{proof}
We proceed by induction on $v \geq 0$. Suppose $v = 0$, so $m = m'$, and consider the Tate spectral sequence %\todo{[reference section on Tate S.S.]?}
\[
E^2 = k[t^{\pm 1}, x]\{y_{m'} , z_{m'} \} \quad \Rightarrow \quad \pi_*(\THH(k) \otimes B(m') )^{t\TT}
\]
By \cref{InfiniteCyclesLemma} the only possible non-zero differentials are those beginning at $y_{m'}$. Furthermore 
\[
d^2(y_{m'}) \overset{.}= t d(y_{m'}) \overset{.}= m' t z_{m'}
\]
 by \cref{SecondTateDifferential} and \cref{HomologyLemma}. Since $m'$ is a unit in $k$, $d^2$ is an isomorphism. The $E^2$-page is summarized in the following diagram (shifted up by $2d$ in the horizontal direction).

\begin{comment}
\begin{sseqdata}[ name = TateS1b, xscale = 0.8, yscale = 0.8,
homological Serre grading , classes = { draw = none } ]
\class["ty_{m'}"](-2,0)
\class(-4,1)
\class["t^2y_{m'}"](-4,0)
\class["t^{-1}y_{m'}"](2,0)
\class["t^{-2}y_{m'}"](4,0)
\class["xy_{m'}"](0,2)
\class["y_{m'}"](0,0)
\class(-2,1)
\class["z_{m'}"](0,1)
\class["xz_{m'}"](0,3)
\class["x^2y_{m'}"](0,4)
\class(0,6)
\class(2,1)
\class["x^2z_{m'}"](0,5)
\class["x^3y_{m'}"](0,6)
\class(-2,3) \class(-2,2) \class(2,2) \class(4,2) \class(-4,3)
\class(-2,5) \class(2,3) \class(2,4) \class(4,4) \class(2,5)
\class(-2,4) \class(-4,5)
\d2(0,0)
\d2(-2,0)
\d2(2,0)
\d2(4,0)
\d2(0,2) \d2(0,4) \d2(-2,2) \d2(2,2) \d2(4,2) \d2(2,4) \d2(4,4) \d2(-2,4)
\end{sseqdata}
\begin{center}
\printpage[name = TateS1b, page = 2] 
\end{center}
\end{comment}

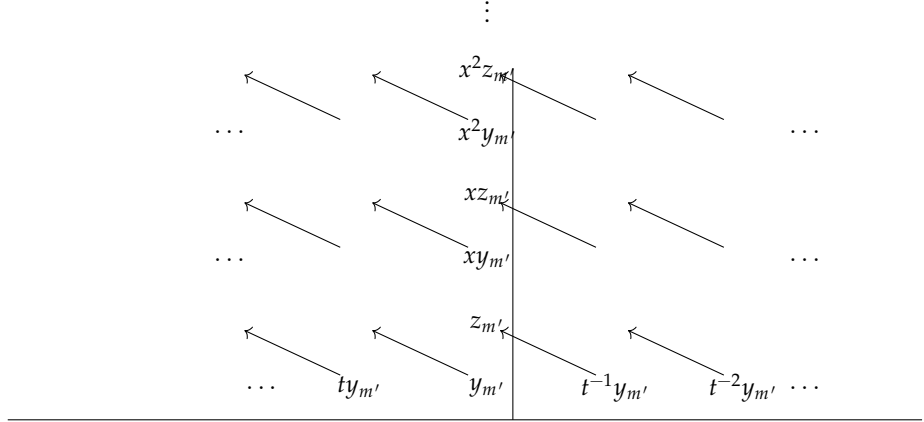
\begin{figure}[h]
\centering
    \begin{tikzpicture}[scale = 0.85, every node/.style={transform shape}]
    \draw (-7.5,-0.5) -- (6.8,-0.5);
    \draw (0.4,-0.5) -- (0.4,5);
    \node at (0,0) {$y_{m'}$};
    \node at (0,1) {$z_{m'}$};
    \node at (0,2) {$xy_{m'}$};
    \node at (0,3) {$xz_{m'}$};
    \node at (0,4) {$x^2y_{m'}$};
    \node at (0,5) {$x^2z_{m'}$};
    \node at (0,6) {$\vdots$};
    %\node at (-1,0) {$u_1$};
    \node at (-2,0) {$ty_{m'}$};
    \node at (2,0) {$t^{-1}y_{m'}$};
    \node at (4,0) {$t^{-2}y_{m'}$};
    \node at (5,0) {$\cdots$};
    %\node at (-3,0) {$tu_1$};
    
    %\node at (1,0) {$t^{-1}u_1$};
   % \node at (1.5,0) {$\cdots$};

    \draw[->] (-2.3,0.2) -> (-3.8,0.9); 
    \draw[->] (-0.3,0.2) -> (-1.8,0.9); % differential on y
    \draw[->] (1.7,0.2) -> (0.2,0.9);   % differential on t^-1y
    \draw[->] (3.7,0.2) -> (2.2,0.9);    
    \draw[->] (-2.3,2.2) -> (-3.8,2.9);
    \draw[->] (-0.3,2.2) -> (-1.8,2.9); % differential on xy
    \draw[->] (1.7,2.2) -> (0.2,2.9);   
    \draw[->] (3.7,2.2) -> (2.2,2.9);    
    \draw[->] (-0.3,4.2) -> (-1.8,4.9);
    \draw[->] (1.7,4.2) -> (0.2,4.9);
    \draw[->] (3.7,4.2) -> (2.2,4.9); 
    \draw[->] (-2.3,4.2) -> (-3.8,4.9);
    \node at (-3.5,0) {$\cdots$};
    \node at (-4,4) {$\cdots$};
    \node at (-4,2) {$\cdots$};
    \node at (5,4) {$\cdots$};
    \node at (5,2) {$\cdots$};
    \end{tikzpicture}
\end{figure}

Thus $E^3 = E^\infty = 0$ is trivial, as claimed. To determine the $\TT$-homotopy fixed points, we truncate the Tate spectral sequence, removing the first quadrant. The classes $x^nz_{m'}$ are no longer hit by differentials and so 
\[
E^3 = E^\infty = k[x]\{z_{m'}\}
\]
where $z_{m'}$ has degree $2d+1$. This proves the claim for $v = 0$.

Suppose the claim is known for all integers less than or equal to $v$. By \cref{FrobeniusCorollary} the Frobenius
\[
\phi(p^vm')^{h\TT} : \pi_* (\THH(k) \otimes B(p^vm'))^{h\TT} \to \pi_* (\THH(k) \otimes B(p^{v+1}m'))^{t\TT}
\]
is an isomorphism in high degrees. The induction hypothesis then implies that the domain is isomorphic to $W_{v+1}(k)$ when $* = 2r+1 \geq 2d+1$.   By periodicity we conclude that $\pi_*(\THH(k) \otimes B(p^{v+1}m'))^{t\TT}$ is concentrated in odd degrees where, 
\[
\pi_{2r+1} (\THH(k) \otimes B(p^{v+1}m'))^{t\TT} \simeq W_{v+1}(k)
\]
for any $r \in \ZZ$. Considering again the Tate spectral sequence we see that we must have
\[
d^{2v + 2}(y_{p^{v+1}m'}) \overset{.}= t (xt)^{v}  z_{p^{v+1}m'}
\]
and so $E^{2v + 3} = E^\infty$.  Truncating the spectral sequence to obtain the homotopy fixed-point spectral sequence, we now see that
\[
\pi_{2r+1}(\THH(k) \otimes B(p^{v+1}m'))^{h\TT} \simeq \left\{
	\begin{array}{ll}
		  W_{v+2}(k) & \mbox{if } d \leq r \\
		  W_{v+1}(k) & \mbox{if } r < d
	\end{array}
\right.
\]
At least up to extension problems. To solve these we note that the  homology class $z_m$ in $B(m)$ provides a map of chain complexes
\[
f_{z_m} : \ZZ[2d+1] \to \ZZ \otimes B(m)
\]
Where $\ZZ[2d+1]$ denotes the chain complex concentrated in degree $2d+1$. We claim that $f_{z_m}$ may in fact be promoted to a map of chain complexes with $\TT$-action, i.e.\ a map in the $\infty$-category $\mathcal{D}(\ZZ)^{B\TT}$. This category is equivalent to the $\infty$-category of mixed complexes over $\ZZ$\footnote{As noted in \cite[Paragraph after Theorem 1]{HesselholtNikolausCusp} this follows from the formality of $C_*(\TT,\ZZ)$ as an $\mathbb{E}_1$-algebra.} hence it suffices to promote $f_{z_m}$ to a map of mixed complexes, i.e.\ a map which commutes with Connes operator. The domain of $f_{z_m}$ is given the trivial $\TT$-action, or equivalently trivial mixed complex structure, with trivial Connes operator. The codomain has the mixed complex structure used in the proof of \cref{HHweightComputation}, which refers back to \cite[Proposition 2.1]{BACH91}. From there we see that Connes operator acts as zero on any representative of the homology class $z_m$.  Thus $f_{z_m}$ may be equipped with the structure of a $\TT$-equivariant map with trivial $\TT$-action on the domain. Upon tensoring with $\THH(k)$ over $\ZZ$ (using the module $\ZZ$-structure given by the map we constructed before \cref{InfiniteCyclesLemma} we get a $\TT$-equivariant map 
\[
\THH(k)[2d+1] \to \THH(k) \otimes B(m) 
\]
In particular this induces a map of Tate spectral sequences. This allows us to conclude that the $W(k)$-module $\pi_{2r+1}(\THH(k) \otimes B(m))^{h\TT}$ is cyclic, hence determined by its length. Indeed, given $\alpha \in \pi_{2r+1}(\THH(k) \otimes B(m))^{h\TT}$ with image $\bar{\alpha}$ on the $E^\infty$-page, it has the form $t^{a}x^az_m$ for some $a \geq 0$ and so is hit by $t^{a}x^a$ on the $E^\infty$-page for $\THH(k)^{t\TT}$, where the extension problem has already been solved. Up to a unit $p^a$ lifts $t^ax^a$. Since the map 
\[
(\THH(k)[2d+1])^{t\TT} \to (\THH(k) \otimes B(m))^{t\TT}
\]
is $W(k)$-linear we see that $\alpha \overset{.}= p^az_m$. 
 This completes the proof.
\end{proof}

To deal with the case where $e$ does divide $m$ we factor $e = p^u e'$ where $(p,e')= 1$. Thus $e \mid m$ if and only if $v \geq u$ and $e' \mid m'$.

\begin{proposition}\label{TPspecial}
Write $m = p^vm'$ where $(m',p)=1$. If $e \mid m$ then 
\[
\pi_{2r+1}(\THH(k) \otimes B(p^vm'))^{t\TT} \simeq W_u(k)
\]
 and
\[
\pi_{2r+1}(\THH(k) \otimes B(p^vm'))^{h\TT} \simeq W_u(k)
\]
for all $r \in \ZZ$.
\end{proposition}
\begin{proof}
If either $u$ or $v$ is zero then $\Htilde_*(B(m);k)$ is trivial in every degree by \cref{HomologyLemma}. The result easily follows. For the rest of the cases we use induction on $v \geq u \geq 1$. Suppose $v = u$. Then 
\[
\pi_*(\THH(k) \otimes B(p^{v-1}m') )^{h\TT} \overset{\phi(p^{v{-}1}m')}{\longto} \pi_*(\THH(k) \otimes B(p^vm') )^{t\TT}
\]
is an isomorphism in high enough degrees. The domain was evaluated in \cref{TPordinary}, it is $W_u(k)$ in odd degrees greater than or equal to $2d+1$. By periodicity we conclude the result for the codomain. 
Now suppose the result has been verified for all integers greater than $u$ and strictly less than $v$. Again using the  Frobenius we conclude that 
\[
\pi_{2r+1}(\THH(k) \otimes B(p^vm'))^{t\TT} \simeq W_u(k)
\]
for all $r \in \ZZ$. 

\begin{comment}
\begin{sseqdata}[ name = TateS1projective, xscale = 1, yscale = 1,
homological Serre grading , classes = { draw = none } ]
\class["t"](-2,0)
\class(-4,2) \class(4,1) \class(-2,6) \class(-4,4) \class(4,3) \class(-4,6) \class(4,5) \class(2,6)
\class["t^2"](-4,0)
\class["t^{-1}"](2,0)
\class["t^{-2}"](4,0)
\class["xz_m"](0,2)
\class["z_{m}"](0,0)
\class(-2,1)
\class["w_{m}"](0,1)
\class["xw_m"](0,3)
\class["x^2z_m"](0,4)
\class(0,6)
\class(2,1)
\class["z_{m'}x^2"](0,5)
\class["x^3"](0,6)
\class(-2,3) \class(-2,2) \class(2,2) \class(4,2) \class(-4,3)
\class(-2,5) \class(2,3) \class(2,4) \class(4,4) \class(2,5)
\class(-2,4) \class(-4,5)
\d2(0,1)
\d2(-2,1)
\d2(2,1)
\d2(4,1)
 \d2(0,3) \d2(0,5) \d2(-2,3) \d2(2,3) \d2(4,3) \d2(2,5)  \d2(-2,5) \d2(4,5)
\end{sseqdata}
%\printpage[name = TateS1projective, page = 2] 
\end{comment}

Consider the Tate spectral sequence with $E^2$-page $k[t^{\pm 1}, x]\{z_m,w_m\}$.
Since $z_m$ is an infinite cycle the only possible way that this sequence collapses to yield the correct result is if
\[
d^{2u}(w_m) \overset{.}{=}  (tx)^u z_m
\]
Thus $E^{2u+1} = E^\infty$. As before, by truncating the first quadrant, we get the spectral sequence for the homotopy  $\TT$-fixed points whose $E^{2u}$-page clearly shows the result. The extension problem is solved as in the proof of \cref{TPordinary}.
\end{proof}

%%%%%%%%%%%%%%%%%%%%%%%%%%%%%%%%%%%%%%%%
%%%%%%%%%%%%%%%%%%%%%%%%%%%%%%%%%%%%%%%%
%%%%%%%%%%%%%%%%%%%%%%%%%%%%%%%%%%%%%%%%

%%%%%%%%%%%%%%%%%%%%%%%%%%%%%%%%%%%%%%%%
%%%%%%%%%%%%%%%%%%%%%%%%%%%%%%%%%%%%%%%%
%%%%%%%%%%%%%%%%%%%%%%%%%%%%%%%%%%%%%%%%

\section{Topological cyclic homology}
\noindent
We now prove \cref{KofTruncatedPoly}.
By McCarthy's Theorem \cite{McCarthy} it suffices to prove the following.

\begin{theorem}\label{TCofTruncatedPoly}
Let $k$ be a perfect field of positive characteristic. Then there is an isomorphism
\[
\TC_{2r-1}(k[x]/(x^e),(x)) \simeq \WW_{re}(k) / V_e \WW_r(k)
\]
and the groups in even degrees are zero.
\end{theorem}
\begin{proof}
In view of \cref{WittSplit} it suffices to give an isomorphism
\[
\TC_{2r-1}(k[x]/(x^e),(x)) \simeq  \prod W_h(k) 
\]
where the product is indexed over $1 \leq m' \leq re$ with $(p,m')=1$ and with $h = h(p,r,e,m')$ given by
\[
h =  \left\{
	\begin{array}{ll}
		  s & \mbox{if } e' \nmid m' \\
		  \min\{u,s\} & \mbox{if } e' \mid m'
	\end{array}
\right.
\]
where $s = s(p,re,m')$ is such that $p^{s-1}m' \leq re < p^sm'$, and where $e = p^ue'$ with $(e',p)=1$.
Now $\TC(A,I)$ is given as the equalizer of $\TC^-(A,I) \overset{\phi - \can}\longto \TP(A,I)$. This  map splits as
\[
\prod_{\substack{ m' \geq 1 \\ (p,m') = 1 } } \prod_{v \geq 0} \TC^-(p^vm') \overset{\phi - \can}\longto \prod_{\substack{ m' \geq 1 \\ (p,m') = 1 } } \prod_{v \geq 0} \TP(p^vm')
\]
By \cref{TPordinary} and \cref{TPspecial} both $\TC^-(p^vm')$ and $\TP(p^vm')$ are concentrated in odd degrees, and $\phi-\can$ is surjective on homotopy, so the long exact sequence calculating $\TC$ splits into short exact sequences
\[
0 \to \TC_*(m') \to \prod_{v \geq 0} \TC^-_*(p^vm') \overset{\phi - \can}\longto  \prod_{v \geq 0} \TP_*(p^vm') \to 0
\]
Now if $e' \nmid m'$ then from \cref{TPordinary} we have a map of short exact sequences
\[
\begin{tikzcd}[row sep=2em, column sep=1.8em]
0 \arrow[r] & \prod_{v \geq s} W_{v}(k) \arrow[d, "\phi-\can"] \arrow[r] & \prod_{v \geq 0} \TC^{-}_{2r+1}(p^vm') \arrow[d, "\phi - \can"] \arrow[r] & \prod_{0 \leq v < s} W_{v+1}(k) \arrow[d, "\overline{\phi - \can}"]  \arrow[r] & 0\\
0 \arrow[r] & \prod_{v \geq s} W_{v}(k) \arrow[r] & \prod_{v \geq 0} \TP_{2r+1}(p^vm') \arrow[r] & \prod_{0 \leq v < s}  W_{v}(k) \arrow[r] & 0
\end{tikzcd}
\]
where $s = s(p,r,d(m'))$. The left hand vertical map is an isomorphism (since in this range $\can$ is an isomorphism and $\phi$ is divisible by powers of $p$) and the right hand vertical map is an epimorphism with kernel $W_{s}(k)$. Thus $\TC_{2r+1}(m') = W_s(k)$. Note that in this case $h = s$.

If $e' \mid m'$ then we must distinguish between two cases. First, if $s < u$ then again we get a map of short exact sequences
\[
\begin{tikzcd}[row sep=2em, column sep=1.4em]
 \prod_{s \leq v < u} W_{v}(k) \times \prod_{u \leq v} W_u(k) \arrow[d, "\phi-\can"] \arrow[r] & \prod_{v \geq 0} \TC^{-}_{2r+1}(p^vm') \arrow[d, "\phi - \can"] \arrow[r] & \prod_{0 \leq v < s} W_{v+1}(k) \arrow[d, "\overline{\phi - \can}"]  \\
 \prod_{s \leq v < u} W_{v}(k) \times  \prod_{u \leq v} W_u(k) \arrow[r] &\prod_{v \geq 0} \TP_{2r+1}(p^vm') \arrow[r] & \prod_{0 \leq v < s}  W_{v}(k) 
\end{tikzcd}
\]
so in this case $\TC_{2r+1}(m') = W_s(k)$ Since $u > s$ we have $h = s$ as claimed.
If instead, $u \leq s$ then the map of short exact sequences looks as follows
\[
\begin{tikzcd}[row sep=2em, column sep=1.5em]
 \prod_{v \geq s} W_{v}(k) \arrow[d, "\phi-\can"] \arrow[r] & \prod_{v \geq 0} \TC^{-}_{2r+1}(p^vm') \arrow[d, "\phi - \can"] \arrow[r] & \prod_{0 \leq v < u} W_{v+1}(k) \times \prod_{u \leq v < s} W_u(k) \arrow[d, "\overline{\phi - \can}"] \\
 \prod_{v \geq s} W_{v}(k) \arrow[r] & \prod_{v \geq 0} \TP_{2r+1}(p^vm') \arrow[r] & \prod_{0 \leq v < u}  W_{v}(k) \times \prod_{u \leq v < s} W_u(k)
\end{tikzcd}
\]
so in this case $\TC_{2r+1}(m') = W_u(k)$. Since $u \leq s$ we see that $u = h$ in this case. This completes the proof.
\end{proof}

%%%%%%%%%%%%%%%%%%%%%%%%%%%%%%%%%%%%%%%%
%%%%%%%%%%%%%%%%%%%%%%%%%%%%%%%%%%%%%%%%
%%%%%%%%%%%%%%%%%%%%%%%%%%%%%%%%%%%%%%%%
\begin{comment}
\section{Algebras with one generator}
\noindent
In this section we consider how one may extend the calculation of the $K$-theory of truncated polynomial algebras to more general algebras of the form $k[x]/(f)$ for some polynomial $f$. The approach is motivated by similar considerations in \cite[Section 3]{BACH91}.

Let $A = k[x]/(f)$ where $f= f(x)$ is a monic polynomial. 

Note if $f$ is monic and irreducible, then $A = k[x]/f = E$ is a field extension of $k$. If for example $k$ is a finite field then this completely determines the $K$-theory by Quillen's work [reference: K-theory of finite fields]. 
\end{comment}

%\nocite{*}
\bibliographystyle{siam}
\bibliography{bibfileNewB}

\begin{thebibliography}{10}

\bibitem{BlumbergMandell2017}
{\sc A.~J. {Blumberg} and M.~A. {Mandell}}, {\em {The strong {K}\"unneth
  theorem for topological periodic cyclic homology}}, arXiv:1706.06846v1,
  (2017).

\bibitem{BHM}
{\sc M.~B\"okstedt, W.~C. Hsiang, and I.~Madsen}, {\em The cyclotomic trace and
  algebraic {$K$}-theory of spaces}, Invent. Math., 111 (1993), pp.~465--539.

\bibitem{BACH91}
{\sc J.~A. Guccione, J.~J. Guccione, M.~J. Redondo, A.~Solotar, and O.~E.
  Villamayor}, {\em Cyclic homology of algebras with one generator}, K-theory,
  5 (1991), pp.~51--69.

\bibitem{HedenlundKrauseNikolaus}
{\sc A.~Hedenlund, A.~Krause, and T.~Nikolaus}, {\em Convergence of spectral
  sequences revisited}, forthcoming.

\bibitem{Hesselholt96}
{\sc L.~Hesselholt}, {\em On the {$p$}-typical curves in {Q}uillen's
  {$K$}-theory}, Acta Math., 177 (1996), pp.~1--53.

\bibitem{HesselholtCusps}
\leavevmode\vrule height 2pt depth -1.6pt width 23pt, {\em On the {K}-theory of
  planar cuspical curves and a new family of polytopes}, Algebraic Topology:
  Applications and New Directions, 620 (2014), p.~145.

\bibitem{HesselholtBigDRW}
\leavevmode\vrule height 2pt depth -1.6pt width 23pt, {\em The big de
  {R}ham--{W}itt complex}, Acta Mathematica, 214 (2015), pp.~135--207.

\bibitem{HesselholtMadsenCyclicPoly}
{\sc L.~Hesselholt and I.~Madsen}, {\em Cyclic polytopes and the {$ K $}-theory
  of truncated polynomial algebras}, Invent. Math., 130 (1997), pp.~73--97.

\bibitem{HesselholtMadsen97}
\leavevmode\vrule height 2pt depth -1.6pt width 23pt, {\em On the {$K$}-theory
  of finite algebras over {W}itt vectors of perfect fields}, Topology, 36
  (1997), pp.~29--101.

\bibitem{HesselholtMadsen2003}
\leavevmode\vrule height 2pt depth -1.6pt width 23pt, {\em On the {$K$}-theory
  of local fields}, Annals of Math., 158 (2003), pp.~1--113.

\bibitem{HesselholtNikolausCusp}
{\sc L.~Hesselholt and T.~Nikolaus}, {\em Algebraic {K}-theory of planar
  cuspidal curves}, K-Theory in Algebra, Analysis, and Topology (Buenos Aires,
  Argentina, 2018).

\bibitem{LarsenLindenstrauss}
{\sc M.~Larsen and A.~Lindenstrauss}, {\em Cyclic homology of dedekind
  domains}, K-theory, 6 (1992), pp.~301--334.

\bibitem{McCarthy}
{\sc R.~McCarthy}, {\em Relative algebraic {$K$}-theory and topological cyclic
  homology}, Acta Mathematica, 179 (1997), pp.~197--222.

\bibitem{NikolausScholze}
{\sc T.~Nikolaus and P.~Scholze}, {\em On topological cyclic homology}, Acta
  Math., 221 (2018), pp.~203--409.

\bibitem{SpeirsCoordinateAxes}
{\sc M.~Speirs}, {\em On the {$K$}-theory of coordinate axes in affine space},
  arXiv:1901.08550,  (2019).

\end{thebibliography}

%\vspace*{\fill}
%\flushright \emph{And you may ask yourself, how do I work this?}

\end{document}